\documentclass{amsart}

\usepackage[section]{placeins}

\usepackage{amsmath}             
\usepackage{amscd}
\usepackage{amssymb}             
\usepackage{amsfonts}            
\usepackage{latexsym}            
\usepackage{amsthm}              
\usepackage{graphicx}

\usepackage{enumerate}
\usepackage{mathrsfs}

\newcommand{\Q}{\mathbb{Q}}   
\newcommand{\F}{\mathbb{F}}   
\newcommand{\K}{\mathbb{K}}

\newcommand{\GL}{\operatorname{GL}}
\newcommand{\GML}{\operatorname{\Gamma L}}

\newcommand{\chr}{\operatorname{char}}

\newcommand{\Gal}{\operatorname{Gal}}

\newtheorem{lause}{Theorem}[section]

\newtheorem{lemma}[lause]{Lemma}
\newtheorem{seur}[lause]{Corollary}

\theoremstyle{remark}
\newtheorem{remark}[lause]{Remark}
 
\newtheorem*{acknow*}{Acknowledgements}

\allowdisplaybreaks

\begin{document}

\title[Orthogonal representations of finite solvable groups]{Orthogonal irreducible representations of finite solvable groups in odd dimension}

\author{Mikko Korhonen}
\address{SUSTech International Center for Mathematics, Southern University of Science and Technology, \text{Shenzhen} 518055, Guangdong, P. R. China}
\email{korhonen\_mikko@hotmail.com}
\thanks{Support by Shenzhen Science and Technology Program (Grant No. RCBS20210609104420034).}

\date{\today}

\begin{abstract}
We prove that if $G$ is a finite irreducible solvable subgroup of an orthogonal group $O(V,Q)$ with $\dim V$ odd, then $G$ preserves an orthogonal decomposition of $V$ into $1$-spaces. In particular $G$ is monomial. This generalizes a theorem of Rod Gow.
\end{abstract}

\vspace*{-18ex}
\maketitle

\vspace*{-5ex}
\section{Introduction}

In \cite{Gow1975}, Rod Gow proved that if $\chi$ is a real-valued complex irreducible character of odd degree of a finite solvable group $G$, then $\chi$ is induced from a linear character of a subgroup. Furthermore, in this case $\chi$ is afforded by a monomial representation where the non-zero entries of the matrices are $\pm 1$, and $\chi$ is rational-valued.

For characters of odd degree, it is well known that a complex irreducible character of $G$ is real-valued if and only if the corresponding representation admits a non-degenerate $G$-invariant quadratic form \cite[Theorem 23.16]{JamesLiebeck2001}. The purpose of this note is to generalize Gow's theorem to other fields from this point of view. Previously in \cite{Vallejo}, the result of Gow was extended to characters which take values in certain cyclotomic extensions of $\Q$.

Let $V$ be a finite-dimensional vector space over a field $\F$ with $\chr \F \neq 2$. Denote $n = \dim V$. Let $Q: V \rightarrow \F$ be a non-degenerate quadratic form with polarization $b: V \times V \rightarrow \F$. Denote the isometry group of $Q$ by $O(V,Q)$, so $$O(V,Q) = \{ g \in \GL(V) : Q(gv) = Q(v) \text{ for all } v \in V\}.$$ We will prove the following.

\begin{lause}\label{thm:mainthm}
Suppose that $n$ is odd. Let $G \leq O(V, Q)$ be finite irreducible solvable. Then there exists an orthogonal decomposition $$V = W_1 \perp \dots \perp W_n$$ such that $\dim W_i = 1$ for all $1 \leq i \leq n$ and $G$ acts on $\{W_1, \ldots, W_n\}$.
\end{lause}

\begin{remark}
The assumptions in Theorem \ref{thm:mainthm} are necessary. If $n$ is even, for any prime power $q > 3$ the orthogonal group $O_2^{-}(q)$ is solvable, irreducible, and primitive.

Furthermore, if $n$ is odd and $G \leq \GL(V)$ is finite irreducible solvable, we need to assume that $G$ is contained in an orthogonal group for Theorem \ref{thm:mainthm} to hold. For example, the semilinear group $\GML_1(q^n)$ is a solvable irreducible primitive subgroup of $\GL_n(q)$ for any prime power $q$.

The assumption that $G$ is solvable is also necessary. For example, the orthogonal group $O_3(q)$ is irreducible and primitive for any odd prime power $q > 3$. 
\end{remark}

\begin{remark}
We have assumed that $\chr \F \neq 2$. In the case where $\chr \F = 2$ and $n$ is odd, the radical $V^\perp$ of $b$ is non-zero, so there are no irreducible subgroups in $O(V,Q)$.
\end{remark}

In $O(V,Q)$, the stabilizer of an orthogonal decomposition $V = W_1 \perp \cdots \perp W_n$ is a wreath product $O_1(\F) \wr S_n$, which after a choice of basis consists of monomial matrices with non-zero entries in $\pm 1$. Theorem \ref{thm:mainthm} states that for $n$ odd any finite irreducible solvable subgroup embeds into $O_1(\F) \wr S_n$, and as a corollary we can classify the maximal finite irreducible solvable subgroups of $O(V,Q)$. This result fits into the general problem of classifying maximal irreducible solvable subgroups of linear groups, as studied for example in \cite[Chapter 5]{SuprunenkoBook}. Similar themes are also found in the classification of maximal subgroups of classical groups, for example as in \cite{Aschbacher1984}, \cite{KleidmanLiebeck}, and \cite{BrayHoltRoneyDougal2013}.

\begin{seur}\label{cor:cor1}
Suppose that $n$ is odd. Let $G$ be maximal among the finite irreducible solvable subgroups of $O(V,Q)$. Then $G$ is conjugate to $O_1(\F) \wr K$ for some maximal transitive solvable $K \leq S_n$.
\end{seur}

\begin{seur}\label{cor:cor2}
Suppose that $n$ is odd, and let $K \leq S_n$ be maximal transitive solvable. Then $O_1(\F) \wr K$ is maximal among the finite irreducible solvable subgroups of $O(V,Q)$.
\end{seur}

\section{Proofs}

\begin{lemma}\label{lemma:mainlemma}
Suppose that $n > 1$ is odd. Let $G \leq O(V, Q)$ be finite irreducible solvable. Then there exists an orthogonal decomposition $$V = W_1 \perp \dots \perp W_k$$ such that $k > 1$ and $G$ acts on $\{W_1, \ldots, W_k\}$.
\end{lemma}

\begin{proof}
For the sake of contradiction, suppose that $G$ preserves no such orthogonal decomposition. 

Let $N \neq 1$ be an abelian normal subgroup of $G$. By Clifford's theorem $V \downarrow N$ is completely reducible and $G$ acts on the homogeneous components of $V \downarrow N$. Since $G$ preserves no non-trivial orthogonal decomposition of $V$, by Clifford theory \cite[Proposition 5]{Zalesski1971} one of the following holds:

	\begin{itemize}
		\item $V \downarrow N$ is homogeneous.
		\item $V \downarrow N = W \oplus W^*$, where $W$ is a homogeneous $\F[N]$-module.
	\end{itemize}

The second case is not applicable since $n$ is odd, so $V \downarrow N$ must be homogeneous. A finite abelian group with a faithful irreducible representation is cyclic, so $N$ must be cyclic.

Let $p = \chr \F$. We claim that $\gcd(p,|N|) = 1$ if $p > 0$. To this end, let $P$ be a Sylow $p$-subgroup of $N$. Then $P$ is normal in $G$, so $G$ acts on the fixed point space $V^P$. Because $P$ is a $p$-group we have $V^P \neq 0$, so $V^P = V$ by irreducibility of $G$, in which case $P$ is trivial. Hence $\gcd(p, |N|) = 1$.

Choose a generator $f$ of $N$. Let $\K/\F$ be the splitting field of the representation of $N$ on $V$, in other words, a splitting field for the characteristic polynomial of $f$. Note that since $\gcd(p, |f|) = 1$, the extension $\K/\F$ is Galois.

Denote $V' := \K \otimes_{\F} V$. We identify $V$ as a $\F$-subspace of $V'$ via $v \mapsto 1 \otimes v$. Then $\GL(V) \leq \GL(V')$, and by extending the bilinear form $b$ linearly to $V'$, we have $O(V,Q) \leq O(V',Q)$. 

Because $\gcd(p, |f|) = 1$, the action of $f$ on $V'$ is diagonalizable, so $V'$ is a direct sum of $f$-eigenspaces. For $\alpha \in \K$, denote by $W_{\alpha}$ the $f$-eigenspace corresponding to $\alpha$. To each element $\sigma \in \Gal(\K/\F)$ of the Galois group, we have a semilinear map $s_{\sigma}: V' \rightarrow V'$ defined by $s_{\sigma} := \sigma \otimes I_V$. Then $s_{\sigma}(W_{\alpha}) = W_{\sigma(\alpha)}$, so in particular $\dim_{\K} W_{\sigma(\alpha)} = \dim_{\K} W_{\alpha}$ for all $\sigma \in \Gal(\K/\F)$ and $\alpha \in \K$.

Denote $$Q_{\alpha} = \sum_{\sigma \in \Gal(\K/\F)} W_{\sigma(\alpha)}$$ for all $\alpha \in \K$. Since $Q_{\alpha}$ is invariant under the action of the Galois group, we have $Q_{\alpha} = \K \otimes_{\F} Z_{\alpha}$ for some $f$-invariant subspace $Z_{\alpha}$ of $V$ \cite[V, Hilfssatz 13.2]{HuppertBook}. 

We have a direct sum decomposition $V' = \sum_{\alpha \in \K} Q_{\alpha}$. Since $V \downarrow N$ is homogeneous, the restriction $V' \downarrow N$ is the sum of a homogeneous $\K[N]$-module and its conjugates under the action of $\Gal(\K/\F)$ \cite[V, Satz 13.3]{HuppertBook}. Thus $V' = Q_{\alpha}$ and $V = Z_{\alpha}$ for some $\alpha \in \K$, so the eigenvalues of $f$ on $V'$ are conjugate under the action of $\Gal(\K/\F)$.

The bilinear form $b$ on $V'$ is non-degenerate, so there exists some $\beta \in \K$ such that $b(W_{\alpha}, W_{\beta}) \neq 0$. Because $b$ is $f$-invariant, we have $\alpha\beta = 1$. If $\alpha \neq \beta$, then since $\Gal(\K/\F)$ acts on pairs $\{\alpha, \beta\}$ with $\alpha\beta = 1$, we have $$V' = Q_{\alpha} = W_{\alpha_1} \oplus W_{\alpha_1^{-1}} \oplus \cdots \oplus W_{\alpha_t} \oplus W_{\alpha_t^{-1}}$$ for some $\alpha_1$, $\ldots$, $\alpha_t$ with $\alpha_i \neq \alpha_i^{-1}$ for all $1 \leq i \leq t$. But since $\dim_{\K} W_{\alpha} = \dim_{\K} W_{\sigma(\alpha)}$ for all $\sigma \in \Gal(\K/\F)$, it follows that $n = \dim_{\K}(V)$ is even, which is a contradiction. Therefore $\alpha = \beta$, so $\alpha^2 = 1$, which implies $\alpha = -1$ since $N$ is non-trivial. 

Thus we have proven that if $N$ is a nontrivial abelian normal subgroup of $G$, then $|N| = 2$ and $N = \langle -I_V \rangle$. In particular $G$ must be non-abelian, as otherwise $G = N$ and $G$ cannot be irreducible since $n > 1$. 

We can now finish the proof similarly to \cite{Gow1975}. Because $G$ is non-abelian, there is a non-trivial abelian normal subgroup $N$ that is contained in $[G,G]$. This is a contradiction, because commutators have determinant $1$, while $\det(-I_V) = (-1)^n = -1$ when $n$ is odd.\end{proof}

\begin{proof}[Proof of Theorem \ref{thm:mainthm}]
For $n = 1$ the result is immediate. We suppose then that $n > 1$ and proceed by induction on $n$. By Lemma \ref{lemma:mainlemma}, there exists an orthogonal decomposition $$V = Z_1 \perp \cdots \perp Z_k,$$ where $1 < k \leq n$ and $G$ acts on $\{Z_1, \ldots, Z_k\}$. Then $n = dk$, where $d = \dim_{\F} Z_i$ for all $1 \leq i \leq k$. 

Since $G$ is irreducible, the action of $N_G(Z_1)$ on $Z_1$ must be irreducible \cite[Theorem 15.1]{SuprunenkoBook}. Thus it follows by induction that there is an orthogonal decomposition $$Z_1 = W_1^{(1)} \perp \cdots \perp W_d^{(1)}$$ such that $N_G(Z_1)$ acts on $\{W_1^{(1)}, \ldots, W_d^{(1)}\}$, and $\dim W_j^{(1)} = 1$ for all $1 \leq j \leq d$. 

Because $G$ is irreducible, it acts transitively on $\{Z_1, \ldots, Z_k\}$, so for all $1 < i \leq k$ there exists $g_i \in G$ such that $g_i(Z_1) = Z_i$. We define $W_j^{(i)} := g_i(W_j^{(1)})$ for all $1 < i \leq k$ and $1 \leq j \leq d$. Then it is straightforward to deduce (see for example \cite[Lemma 15.2]{SuprunenkoBook}) that $$Z_i = W_1^{(i)} \perp \cdots \perp W_d^{(i)}$$ for all $1 \leq i \leq k$ and $G$ acts on $\{W_1^{(1)}, \ldots, W_d^{(1)}, \ldots, W_1^{(k)}, \ldots, W_d^{(k)}\}$. This completes the proof of the theorem.\end{proof}

\begin{proof}[Proof of Corollary \ref{cor:cor1}]Let $G \leq O(V,Q)$ be maximal among the finite irreducible solvable subgroups of $O(V,Q)$. It follows from Theorem \ref{thm:mainthm} that $G$ is conjugate to a subgroup of $O_1(\F) \wr S_n$, so suppose that $G \leq O_1(\F) \wr S_n$. Then $G \leq O_1(\F) \wr K$, where $K$ is the image of $G$ in $S_n$. Since $G$ is irreducible and solvable, the group $K$ must be transitive and solvable. By maximality of $G$, we have $G = O_1(\F) \wr K$ and $K$ is maximal transitive solvable.\end{proof}

\begin{proof}[Proof of Corollary \ref{cor:cor2}]Denote $G = O_1(\F) \wr K$, and let $$V = W_1 \perp \cdots \perp W_n$$ be the orthogonal decomposition defining $G$. If $G$ is not maximal solvable, by Corollary \ref{cor:cor1} there exists $G' \leq O(V,Q)$ of the form $G' = O_1(\F) \wr K'$ with $K' \leq S_n$ maximal transitive solvable and $G \lneqq G'$. Let $$V = Z_1 \perp \cdots \perp Z_n$$ be the orthogonal decomposition defining $G'$. Since $G$ acts on $\{Z_1, \ldots, Z_n\}$ and $n$ is odd, it follows from \cite[Theorem 1.1]{KorhonenLi2021} that $\{W_1, \ldots, W_n\} = \{Z_1, \ldots, Z_n\}$. Thus we have an embedding $K \leq K'$, so $K = K'$ since $K$ is maximal solvable. But then $G = G'$, contradicting our assumption $G \lneqq G'$.\end{proof}

\end{document}